\documentclass[a4paper,10pt]{article}

\usepackage{amsmath,amssymb,amsthm,amscd}
\usepackage[mathscr]{eucal}
\usepackage{multirow}
\usepackage{hhline}
\usepackage[all]{xy}

\makeatletter
 
  \@addtoreset{equation}{subsection}
\makeatother

\theoremstyle{plain}

\newcommand{\plim}{\varprojlim}
\newcommand{\mcal}{\mathcal}
\newcommand{\mbf}{\mathbf}

\newcommand{\mfrak}{\mathfrak}
\newcommand{\mbb}{\mathbb}
\newcommand{\mrm}{\mathrm}
\newcommand{\mf}{\mathfrak}
\newcommand{\mfS}{\mathfrak{S}}
\newcommand{\vphi}{\varphi}
\newcommand{\mfM}{\mathfrak{M}}

\newcommand{\wh}{\widehat}
\newcommand{\whR}{\widehat{\mathcal{R}}}
\newcommand{\whRi}{\widehat{\mathcal{R}}_{\infty}}

\newtheorem{theorem}{Theorem}[section]
\newtheorem{corollary}[theorem]{Corollary}
\newtheorem{lemma}[theorem]{Lemma}

\newtheorem{proposition}[theorem]{Proposition}

\theoremstyle{definition}
\newtheorem{definition}[theorem]{Definition}
\newtheorem{remark}[theorem]{Remark}

\newtheorem{example}[theorem]{Example}

\newcommand{\e}{\varepsilon}
\newcommand{\E}{\mathcal{E}}
\newcommand{\cO}{\mathcal{O}}

\title{Cartier duality for $(\varphi, \hat{G})$-modules}

\author{Yoshiyasu Ozeki\footnote{
Graduate School of Mathematics, 
Kyushu University, Fukuoka 819-0395, Japan.
\endgraf
e-mail: {\tt y-ozeki@math.kyushu-u.ac.jp}}}

\date{}
\begin{document}
\maketitle

\begin{abstract}
In this paper, 
we prove the Cartier duality  
for $(\vphi, \hat{G})$-modules
which are defined by Tong Liu to classify semistable Galois representations.
\end{abstract}

\section{Introduction}
Let $k$ be a perfect field of 
characteristic $p\ge 2$,
$W(k)$ its ring of Witt vectors, 
$K_0:=W(k)[1/p]$, $K$ a finite totally 
ramified extension of $K_0$,
$\bar{K}$ a fixed algebraic closure of $K$
and $G:=\mrm{Gal}(\bar K/K)$.
Let $r\ge 0$ be an integer.
In this paper, 
we give the Cartier duality  
for $(\varphi, \hat{G})$-modules.
A $(\varphi, \hat{G})$-module is a Kisin module equipped with 
certain Galois action, which was introduced by Liu in \cite{Li3}
to classify lattices in 
semistable representations of $G$ whose
Hodge-Tate weights are in $[0,r]$.

Breuil defined strongly divisible lattices and 
conjectured that, if $r<p-1$, 
the category of 
strongly divisible lattices of weight $r$
is equivalent to the category $\mrm{Rep}^r_{\mbb{Z}_p}(G)$  of 
$G$-stable $\mbb{Z}_p$-lattices in semistable $p$-adic representations of $G$ 
whose Hodge-Tate weights are in $[0,r]$.
This conjecture was proved by Liu in \cite{Li2} if $p\ge 3$.
To remove the condition $r<p-1$, Liu defined a free $(\varphi, \hat{G})$-module 
of height $r$ in \cite{Li3}
and proved that there exists a functor $\hat{T}$
from the category
of free $(\varphi, \hat{G})$-modules of height $r$
into $\mrm{Rep}^r_{\mbb{Z}_p}(G)$ 
which induces an equivalence of those categories
for any $p\ge 2$ and $r\ge 0$.
The notion of torsion $(\varphi, \hat{G})$-modules 
of height $r$ is given in \cite{CL} and 
there exists a functor $\hat{T}$ 
from the category of torsion 
$(\varphi, \hat{G})$-modules of weight $r$
into torsion $\mbb{Z}_p$-representations of $G$
as in the case of free $(\varphi, \hat{G})$-modules.

The Cartier duality for
Breuil modules and Kisin modules 
have been studied by Caruso \cite{Ca2} and Liu \cite{Li1}, respectively.
Let $\hat{\mfrak{M}}$ be a torsion (resp.\ free) 
$(\varphi, \hat{G})$-module of height $r$.
We define {\it the dual 
$\hat{\mfrak{M}}^{\vee}$} of 
$\hat{\mfrak{M}}$ in Section 3, which depends on $r$.
Our main result is

\begin{theorem}
\label{Dual}
Let $p\ge 2$ be a prime number and $r\ge 0$ an integer.
Let $\hat{\mfrak{M}}$ be a torsion (resp.\ free) 
$(\varphi, \hat{G})$-module of height $r$ and
$\hat{\mfrak{M}}^{\vee}$ its dual.

\noindent
$(1)$
The dual $\hat{\mfrak{M}}^{\vee}$ 
is a torsion (resp.\ free) 
 $(\varphi, \hat{G})$-module of height $r$.

\noindent
$(2)$ The assignment 
$\hat{\mfrak{M}}\mapsto \hat{\mfrak{M}}^{\vee}$
is an anti-equivalence on the category 
of torsion (resp.\ free)
$(\varphi, \hat{G})$-modules 
and a natural map 
$\hat{\mfrak{M}}\to (\hat{\mfrak{M}}^{\vee})^{\vee}$
is an isomorphism.

\noindent
$(3)$  
$\hat{T}(\hat{\mfrak{M}}^{\vee})\simeq 
\hat{T}^{\vee}(\hat{\mfrak{M}})(r)$ as 
$\mbb{Z}_p$-representations of $G$.
\end{theorem}

\noindent
Here $\hat{T}^{\vee}(\hat{\mfrak{M}})$ is 
the dual representation of  
$\hat{T}(\hat{\mfrak{M}})$
and the symbol ``$(r)$'' in the assertion $(3)$ 
is for the $r$-th Tate twist. 

Now let $p\ge 3$ and $r<p-1$.
Thanks to Liu's theory of \cite{Li2},
we can construct a natural functor  
$\mcal{M}_{\whR}$ from the category of 
free $(\vphi,\hat{G})$-modules of height $r$
into the category of 
strongly divisible lattices of weight $r$,
which gives an equivalence of those categories.
By using this functor, 
we obtain the following comparison result 
between the duality of $(\vphi, \hat{G})$-modules and 
that of Breuil modules.
\begin{corollary}
For a  free 
$(\varphi, \hat{G})$-module  $\hat{\mfrak{M}}$,
there exists a canonical isomorphism 
\[
\mcal{M}_{\whR}(\hat{\mfM}^{\vee})\simeq
\mcal{M}_{\whR}(\hat{\mfM})^{\vee}.
\]
Here  $\mcal{M}_{\whR}(\hat{\mfM})^{\vee}$ is the Cartier dual of 
the strongly divisible lattice $\mcal{M}_{\whR}(\hat{\mfM})$
(cf.\ \cite{Ca2}, Chapter V).
\end{corollary}

\section{Finite $\mbb{Z}_p$-representations}

\subsection{Kisin modules}
Let $k$ be a perfect field of 
characteristic $p\ge 2$,
$W(k)$ its ring of Witt vectors, 
$K_0:=W(k)[1/p]$, $K$ a finite totally 
ramified extension of $K_0$,
$\bar{K}$ a fixed algebraic closure of $K$
and $G:=\mrm{Gal}(\bar K/K)$.
Throughout this paper,
we fix a uniformizer $\pi\in K$ 
and denote by $E(u)$ its 
Eisenstein polynomial over $K_0$.
Let $\mfS:=W(k)[\![u]\!]$ equipped 
with a Frobenius endomorphism
$\varphi$ via $u\mapsto u^p$ and 
the natural Frobenius on $W(k)$. 
A {\it $\vphi$-module} ({\it over $\mfS$}) 
is a $\mfS$-module 
$\mfM$ equipped with a $\vphi$-semilinear map 
$\vphi\colon \mfM\to \mfM$.
A $\vphi$-module is called a {\it Kisin module}.
A morphism between two $\vphi$-modules 
$(\mfM_1,\vphi_1)$ and $(\mfM_2,\vphi_2)$
is a $\mfS$-linear morphism 
$\mfM_1\to \mfM_2$ compatible 
with Frobenii $\vphi_1$ and $\vphi_2$.
Denote by $\mrm{Mod}^r_{/\mfS}$
the category of $\vphi$-modules 
{\it of finite $E(u)$-height $r$} ({\it or, of height $r$}) 
in the sense that $\mfM$ is of finite type 
over $\mfS$ and the cokernel of 
$\vphi^{\ast}$ is killed by $E(u)^r$,
where $\vphi^{\ast}$ is the $\mfS$-linearization 
$1\otimes \vphi\colon \mfS\otimes_{\vphi,\mfS}\mfM\to \mfM$
of $\vphi$.
Let $\mrm{Mod}^{r,\mrm{tor}}_{/\mfS}$
be the full subcategory of $\mrm{Mod}^r_{/\mfS}$
consisting of finite $\mfS$-modules $\mfM$
which satisfy the following:

\begin{itemize}
\item $\mfM$ is killed by some power of $p$,
\item $\mfM$ has a two term resolution by finite free
$\mfS$-modules, that is, there exists an exact sequence 
\[
0\to \mf{N}_1\to \mf{N}_2\to \mfM\to 0 
\]
of $\mfS$-modules where 
$\mf{N}_1$ and $\mf{N}_2$ 
are finite free $\mfS$-modules.
\end{itemize}
Let $\mrm{Mod}^{r,\mrm{fr}}_{/\mfS}$
be the full subcategory of 
$\mrm{Mod}^r_{/\mfS}$
consisting of finite free $\mfS$-modules. 
Let $R:=\plim \cO_{\bar K}/p$ 
where $\cO_{\bar K}$ is 
the integer ring of $\bar K$
and the transition maps are 
given by the $p$-th power map.
By the universal property of 
the ring of Witt vectors $W(R)$ of $R$,
there exists a unique surjective projection
map $\theta\colon W(R)\to \wh{\cO}_{\bar K}$
which lifts the projection $R\to \cO_{\bar K}/p$
onto the first factor in the inverse limit,
where $\wh{\cO}_{\bar K}$ is the 
$p$-adic completion of $\cO_{\bar K}$.
For any integer $n\ge 0$,
let $\pi_n\in \bar K$ be 
a $p^n$-th root of $\pi$ such that 
$\pi^p_{n+1}=\pi_n$ and write 
\underbar{$\pi$} $=(\pi_n)_{n\ge 0}\in R$. 
Let $[$\underbar{$\pi$}$]\in W(R)$ 
be the Teichm\"uller
representative of \underbar{$\pi$}.
We embed the $W(k)$-algebra $W(k)[u]$ into $W(R)$
via the map $u\mapsto [$\underbar{$\pi$}$]$.
This embedding extends to an 
embedding $\mfS\hookrightarrow W(R)$, 
which is compatible with Frobenius endomorphisms.

Denote by $\cO_{\E}$ the $p$-adic completion 
of $\mfS[1/u]$.
Then $\cO_{\E}$ is a discrete valuation ring 
with uniformizer $p$ and residue field $k(\!(u)\!)$.  
Denote by $\E$ the field of fractions of $\cO_{\E}$.
The inclusion $\mfS\hookrightarrow W(R)$
extends to inclusions $\cO_{\E}\hookrightarrow W(\mrm{Fr}R)$
and $\E\hookrightarrow W(\mrm{Fr}R)[1/p]$.
Here $\mrm{Fr}R$ is the field of fractions of $R$.
It is not difficult to see that $\mrm{Fr}R$ is algebraically closed.
We denote by $\E^{\mrm{ur}}$ the maximal unramified 
field extension of $\E$ in $W(\mrm{Fr}R)[1/p]$
and $\cO^{\mrm{ur}}$ its integer ring.
Let $\wh{\E^{\mrm{ur}}}$ be the 
$p$-adic completion of $\E^{\mrm{ur}}$ and 
$\wh{\cO^{\mrm{ur}}}$ its integer ring.
The ring  $\wh{\E^{\mrm{ur}}}$ 
(resp.\ $\wh{\cO^{\mrm{ur}}}$)
is equal to the closure of 
$\E^{\mrm{ur}}$ in $W(\mrm{Fr}R)[1/p]$
(resp. the closure of $\cO^{\mrm{ur}}$ in $W(\mrm{Fr}R)$). 
Put $\mfS^{\mrm{ur}}:=\wh{\cO^{\mrm{ur}}}\cap W(R)$.
We regard all these rings as subrings of $W(\mrm{Fr}R)[1/p]$.
Put $\mfS_{\infty}^{\mrm{ur}}
:=\mfS^{\mrm{ur}}[1/p]/\mfS^{\mrm{ur}}$.

Let $K_{\infty}:=\cup_{n\ge 0}K(\pi_n)$ and 
$G_{\infty}:=\mrm{Gal}(\bar K/K_{\infty})$.
Then $G_{\infty}$ acts on $\mfS^{\mrm{ur}}$ 
and $\E^{\mrm{ur}}$ continuously and fixes the subring $\mfS\subset W(R)$.
We denote by 
$\mrm{Rep}_{\mbb{Z}_p}(G_{\infty})$
the category of 
continuous $\mbb{Z}_p$-linear representations 
of $G_{\infty}$
on finite $\mbb{Z}_p$-modules.
We denote by 
$\mrm{Rep}^{\mrm{tor}}_{\mbb{Z}_p}(G_{\infty})$
(resp.\ $\mrm{Rep}^{\mrm{fr}}_{\mbb{Z}_p}(G_{\infty})$)
the full subcategory of 
$\mrm{Rep}_{\mbb{Z}_p}(G_{\infty})$
consisting of 
$\mbb{Z}_p$-modules killed by some power of $p$
(resp.\ finite free $\mbb{Z}_p$-modules).  

For any $\mfM\in \mrm{Mod}^{r,\mrm{tor}}_{/\mfS}$,
we define a $\mbb{Z}_p[G_{\infty}]$-module  
via
\[
T_{\mfS}(\mfM):=\mrm{Hom}_{\mfS,\vphi}(\mfM,\mfS_{\infty}^{\mrm{ur}}),
\]
where a $G_{\infty}$-action on 
$T_{\mfS}(\mfM)$ is given by 
$(\sigma.g)(x):=\sigma(g(x))$ 
for $\sigma\in G_{\infty}, g\in T_{\mfS}(\mfM), x\in \mfM$.
The representation $T_{\mfS}(\mfM)$ is an object of 
$\mrm{Rep}^{\mrm{tor}}_{\mbb{Z}_p}(G_{\infty})$.

Similarly, for any $\mfM\in \mrm{Mod}^{r,\mrm{fr}}_{/\mfS}$,
we define a $\mbb{Z}_p[G_{\infty}]$-module via
\[
T_{\mfS}(\mfM):=\mrm{Hom}_{\mfS,\vphi}(\mfM,\mfS^{\mrm{ur}}).
\]
The representation $T_{\mfS}(\mfM)$ is an object of 
$\mrm{Rep}^{\mrm{fr}}_{\mbb{Z}_p}(G_{\infty})$
and $\mrm{rank}_{\mbb{Z}_p}T_{\mfS}(\mfM)=\mrm{rank}_{\mfS}\mfM$.

\begin{theorem}[\cite{Ki}]
The functor 
$T_{\mfS}\colon  \mrm{Mod}^{r,\mrm{fr}}_{/\mfS}\to 
\mrm{Rep}^{\mrm{fr}}_{\mbb{Z}_p}(G_{\infty})$
is fully faithful.
\end{theorem}
\begin{proof}
The desired assertion follows from 
Corollary (2.1.4) and Proposition (2.1.12) of \cite{Ki}
and Fontaine's theory (below).
\end{proof}

\subsection{Fontaine's theory}
A finite $\cO_{\E}$-module $M$ 
is called {\it \'etale}
if $M$ is equipped with 
a $\vphi$-semi-linear map 
$\vphi_M\colon M\to M$ such that 
$\vphi^{\ast}_M$ is an isomorphism, 
where $\vphi^{\ast}_M$ is 
the $\cO_{\E}$-linearization
$1\otimes \vphi_M\colon \cO_{\E}
\otimes_{\vphi, \cO_{\E}} M\to M$
of $\vphi_M$. 
We denote by $\mbf{\Phi M}_{\cO_{\E}}$ 
the category of finite \'etale $\cO_{\E}$-modules
with the obvious morphisms.
Note that the extension $K_{\infty}/K$ is a strictly 
APF extension in the sense of \cite{Wi}
and thus $G_{\infty}$ is 
naturally isomorphic to the absolute Galois group of $k(\!(u)\!)$
by the theory of norm fields.
Combining this fact and Fontaine's theory 
in \cite{Fo}, A 1.2.6, we have that
the functor
\[
T_{\ast}\colon \mbf{\Phi M}_{\cO_{\E}}\to 
\mrm{Rep}_{\mbb{Z}_p}(G_{\infty}),\quad M\mapsto 
(\wh{\cO^{\mrm{ur}}}\otimes_{\cO_{\E}} M)^{\vphi=1} 
\]
is an equivalence of Abelian categories
and there exists the following  natural 
$\wh{\cO^{\mrm{ur}}}$-linear 
isomorphism which is compatible with $\vphi$-structures
and $G_{\infty}$-actions:
\begin{equation}
\label{Fon1}
\wh{\cO^{\mrm{ur}}}\otimes_{\mbb{Z}_p}T_{\ast}(M) 
\overset{\simeq}{\longrightarrow}
\wh{\cO^{\mrm{ur}}}\otimes_{\cO_{\E}} M.  
\end{equation}
Furthermore, 
the functor $T_{\ast}$ induces equivalences 
of categories between
the category of finite torsion \'etale $\cO_{\E}$-modules 
and $\mrm{Rep}^{\mrm{tor}}_{\mbb{Z}_p}(G_{\infty})$
(resp.\ the category of finite free \'etale $\cO_{\E}$-modules 
and  $\mrm{Rep}^{\mrm{fr}}_{\mbb{Z}_p}(G_{\infty})$).

The contravariant version of the functor $T_{\ast}$
is useful for integral theory.
For any $M\in \mbf{\Phi M}_{\cO_{\E}}$,
put
\[
T(M):=\mrm{Hom}_{\cO_{\E},\vphi}(M,\E^{\mrm{ur}}/\cO^{\mrm{ur}})\quad 
\mrm{if}\ M\ \mrm{is\ killed\ by\ some\ power\ of}\ p 
\]
and 
\[
T(M):=\mrm{Hom}_{\cO_{\E},\vphi}(M,\wh{\cO^{\mrm{ur}}})\quad 
\mrm{if}\ M\ \mrm{is}\ p\ \mrm{torsion\ free}. 
\]
Then we can check that $T(M)$ is 
the dual representation of $T_{\ast}(M)$.

Let $\mfM\in \mrm{Mod}^{r,\mrm{tor}}_{/\mfS}$ 
(resp.\ $\mfM\in \mrm{Mod}^{r,\mrm{fr}}_{/\mfS}$).
Since $E(u)$ is a unit in $\cO_{\E}$,
we see that $M:=\cO_{\E}\otimes_{\mfS} \mfM$ 
is a finite torsion \'etale $\cO_{\E}$-module
(resp.\ a finite free \'etale $\cO_{\E}$-module).
Here a Frobenius structure on $M$ is given by 
$\vphi_{M}:=\vphi_{\cO_{\E}}\otimes \vphi_{\mfM}$. 

\begin{theorem}[\cite{Li1}, Corollary 2.2.2]
Let
$\mfM\in \mrm{Mod}^{r,\mrm{tor}}_{/\mfS}$ or 
$\mfM\in \mrm{Mod}^{r,\mrm{fr}}_{/\mfS}$.
Then the natural map
\[
T_{\mfS}(\mfM)\to T(\cO_{\E}\otimes_{\mfS} \mfM)
\]
is an isomorphism as $\mbb{Z}_p$-representations 
of $G_{\infty}$.
\end{theorem}

\subsection{Liu's theory}
Let $S$ be the $p$-adic completion 
of $W(k)[u,\frac{E(u)^i}{i!}]_{i\ge 0}$ and endow $S$
with the following structures:
\begin{itemize}
\item a continuous $\varphi$-semilinear 
Frobenius $\varphi\colon S\to S$
defined by $\varphi(u):=u^p$.

\item a continuous linear derivation 
      $N\colon S\to S$ defined by $N(u):=-u$.
      
\item a decreasing filtration 
      $(\mrm{Fil}^iS)_{i\ge 0}$ in $S$. Here $\mrm{Fil}^iS$ 
      is the $p$-adic closure of the ideal generated 
      by the divided powers $\gamma_j(E(u))=\frac{E(u)^j}{j!}$ for all $j\ge i$.       

\end{itemize}
Put $S_{K_0}:=S[1/p]=K_0\otimes_{W(k)} S$.
The inclusion $W(k)[u]\hookrightarrow W(R)$ 
via the map $u\mapsto [$\underbar{$\pi$}$]$
induces inclusions 
$\mfS\hookrightarrow S\hookrightarrow A_{\mrm{cris}}$
and $S_{K_0}\hookrightarrow B^+_{\mrm{cris}}$.
We regard all these rings as subrings 
in $B^+_{\mrm{cris}}$.

Fix a choice of primitive $p^i$-root 
of unity $\zeta_{p^i}$ for $i\ge 0$
such that $\zeta^p_{p^{i+1}}=\zeta_{p^i}$.
Put \underbar{$\e$} $:=(\zeta_{p^i})_{i\ge 0}\in R^{\times}$
and $t:=\mrm{log}([$\underbar{$\e$}$])\in A_{\mrm{cris}}$.
Denote by $\nu\colon W(R)\to W(\bar k)$ 
a unique lift of the projection $R\to \bar k$.
Since $\nu(\mrm{Ker}(\theta))$ 
is contained in the set $pW(\bar k)$,
$\nu$ extends to a map 
$\nu\colon A_{\mrm{cris}}\to W(\bar k)$
and $\nu \colon B^+_{\mrm{cris}}\to W(\bar k)[1/p]$.
For any subring $A\subset B^+_{\mrm{cris}}$,
we put 
$I_+A:=\mrm{Ker}(\nu\ \mrm{on}\  B^+_{\mrm{cris}})\cap A$.
For any integer $n\ge 0$,
let $t^{\{n\}}:=t^{r(n)}\gamma_{\tilde{q}(n)}(\frac{t^{p-1}}{p})$ 
where $n=(p-1)\tilde{q}(n)+r(n)$ with $0\le r(n) <p-1$
and $\gamma_i(x)=\frac{x^i}{i!}$ is 
the standard divided power.

We define a subring $\mcal{R}_{K_0}$ of $B^+_{\mrm{cris}}$
as below:
\[
\mcal{R}_{K_0}:=\{\sum^{\infty}_{i=0} f_it^{\{i\}}\mid f_i\in S_{K_0}\
\mrm{and}\ f_i\to 0\ \mrm{as}\ i\to \infty\}.
\]
Put $\wh{\mcal{R}}:=\mcal{R}_{K_0}\cap W(R)$
and $I_+:=I_+\wh{\mcal{R}}$.

For any field  $F$ over $\mbb{Q}_p$, set 
$F_{p^{\infty}}:=\cup^{\infty}_{n\ge 0} F(\zeta_{p^n})$.
Recall $K_{\infty}=\cup_{n\ge 0}K(\pi_n)$ 
and note that 
$K_{\infty, p^{\infty}}=\cup_{n\ge 0} K(\pi_n, \zeta_{p^{\infty}})$
is the Galois closure of 
$K_{\infty}$ over $K$.
Put $H_K:=\mrm{Gal}(K_{\infty,p^{\infty}}/K_{\infty}), 
G_{p^{\infty}}:=\mrm{Gal}(K_{\infty,p^{\infty}}/K_{p^{\infty}})$ and 
$\hat{G}:=\mrm{Gal}(K_{{\infty},p^{\infty}}/K)$.

\begin{proposition}[\cite{Li3}, Lemma 2.2.1]
$(1)$ $\wh{\mcal{R}}$ (resp.\ $\mcal{R}_{K_0}$) 
is a $\vphi$-stable $\mfS$-algebra
as a subring in $W(R)$ (resp.\ $B^+_{\mrm{cris}}$).

\noindent
$(2)$ $\wh{\mcal{R}}$ and $I_+$ 
(resp.\ $\mcal{R}_{K_0}$ and $I_+\mcal{R}_{K_0}$) are $G$-stable.
The $G$-action on $\wh{\mcal{R}}$ and $I_+$ 
(resp.\ $\mcal{R}_{K_0}$ and $I_+\mcal{R}_{K_0}$)
factors through $\hat{G}$.

\noindent
$(3)$ There exist natural isomorphisms 
$\mcal{R}_{K_0}/I_+\mcal{R}_{K_0}\simeq K_0$ and 
$\wh{\mcal{R}}/I_+\simeq S/I_+S\simeq \mfS/I_+\mfS\simeq W(k)$.
\end{proposition}

For any Kisin module $(\mfM,\vphi_{\mfM})$ 
of height $r$,
we put $\hat{\mfM}:=\wh{\mcal{R}}\otimes_{\vphi, \mfS} \mfM$
and equip $\hat{\mfM}$ with 
a Frobenius $\vphi_{\hat{\mfM}}$ by
$\vphi_{\hat{\mfM}}:=
\phi_{\wh{\mcal{R}}}\otimes \vphi_{\hat{\mfM}}$.
It is known that a natural map
\[
\mfM\rightarrow \wh{\mcal{R}}\otimes_{\vphi, \mfS} \mfM=\hat{\mfM}
\]
is an injection (\cite{CL}, Section 3.1).
By this injection, 
we regard $\mfM$ as a $\vphi(\mfS)$-stable 
submodule in $\hat{\mfM}$.

\begin{definition}
\label{Liu}
A {\it $(\vphi, \hat{G})$-module} $(${\it of height} $r$$)$ 
is a triple $\hat{\mfM}=(\mfM, \vphi_{\mfM}, \hat{G})$ where
\begin{enumerate}
\item[(1)] $(\mfM, \vphi_{\mfM})$ is a 
           Kisin module (of height $r$), 

\vspace{-2mm}
           
\item[(2)] $\hat{G}$ is an $\wh{\mcal{R}}$-semi-linear
           $\hat{G}$-action on 
           $\hat{\mfM}:=\wh{\mcal{R}}\otimes_{\vphi, \mfS} \mfM$,
           
\vspace{-2mm}           
           
\item[(3)] the $\hat{G}$-action commutes with $\vphi_{\hat{\mfM}}$,

\vspace{-2mm}

\item[(4)] $\mfM\subset \hat{\mfM}^{H_K}$,

\vspace{-2mm}

\item[(5)] $\hat{G}$ acts on the 
           $W(k)$-module $\hat{\mfM}/I_+\hat{\mfM}$ trivially.           
\end{enumerate}

If $\mfM$ is a torsion 
(resp.\ free) Kisin module of (height $r$),
we call $\hat{\mfM}$ a 
{\it torsion} 
(resp.\ {\it free}) {\it $(\vphi, \hat{G})$-module} 
({\it of} height $r$).
If $\hat{\mfM}=(\mfM,\vphi_{\mfM},\hat{G})$ 
is a $(\vphi, \hat{G})$-module,
we often abuse of notations by 
denoting $\hat{\mfM}$ the underlying 
module $\wh{\mcal{R}}\otimes_{\vphi, \mfS} \mfM$.
\end{definition}

A morphism 
$f\colon (\mfM, \vphi, \hat{G})\to (\mfM', \vphi', \hat{G})$  
between two $(\vphi, \hat{G})$-modules is a morphism 
$f\colon (\mfM,\vphi)\to (\mfM',\vphi')$ of Kisin-modules 
such that 
$\wh{\mcal{R}}\otimes f\colon \hat{\mfM}\to \hat{\mfM}'$
is a $\hat{G}$-equivalent.
We denote by $\mrm{Mod}^{r,\hat{G}}_{\mfS}$ 
(resp.\  $\mrm{Mod}^{r,\hat{G},\mrm{tor}}_{\mfS}$, 
resp.\ $\mrm{Mod}^{r,\hat{G},\mrm{fr}}_{\mfS}$) the 
category of $(\vphi, \hat{G})$-modules (resp.\ 
torsion $(\vphi, \hat{G})$-modules, resp.\ 
free $(\vphi, \hat{G})$-modules).
We regard $\hat{\mfM}$ as a $G$-module 
via the projection $G\twoheadrightarrow \hat{G}$. 

For a $(\vphi, \hat{G})$-module $\hat{\mfM}$,
we define a $\mbb{Z}_p[G]$-module as below:
\[
\hat{T}(\hat{\mfM}):=\mrm{Hom}_{\wh{\mcal{R}},\vphi}(\hat{\mfM}, W(R)_{\infty})\quad 
\mrm{if}\ \mfM \ \mrm{is\ killed\ by\ some\ power\ of}\ p
\]
and 
\[
\hat{T}(\hat{\mfM}):=\mrm{Hom}_{\wh{\mcal{R}},\vphi}(\hat{\mfM}, W(R))\quad 
\mrm{if}\ \mfM \ \mrm{is\ free}.
\]
Here, $W(R)_{\infty}:=W(R)[1/p]/W(R)$ and  $G$ 
acts on $\hat{T}(\hat{\mfM})$ by $(\sigma.f)(x):=\sigma(f(\sigma^{-1}(x)))$
for $\sigma\in G,\ f\in \hat{T}(\hat{\mfM}),\ x\in \hat{\mfM}$. 

Let $\hat{\mfM}=(\mfM, \vphi_{\mfM}, \hat{G})$ 
be a $(\vphi,\hat{G})$-module.
There exists a natural map 
\[
\theta\colon T_{\mfS}(\mfM)\to \hat{T}(\hat{\mfM})
\]
defined by 
\[
\theta(f)(a\otimes m):=a\vphi(f(m))\quad \mrm{for}\ 
f\in T_{\mfS}(\mfM),\ a\in \wh{\mcal{R}}, m\in \mfM,
\]
which is a $G_{\infty}$-equivalent.

\begin{theorem}[\cite{Li3},\cite{CL}]
\label{equiv}
Let $\hat{\mfM}=(\mfM, \vphi_{\mfM}, \hat{G})$ 
be a $(\vphi,\hat{G})$-module.

\noindent
$(1)$ The map
$\theta\colon T_{\mfS}(\mfM)\to \hat{T}(\hat{\mfM})$
is an isomorphism of $\mbb{Z}_p[G_{\infty}]$-modules.

\noindent
$(2)$ The functor $\hat{T}$ induces an anti-equivalence 
between the category $\mrm{Mod}^{r,\hat{G},\mrm{fr}}_{\mfS}$ 
of free $(\vphi, \hat{G})$-modules of height $r$
and the category $\mrm{Rep}^{r}_{\mbb{Z}_p}(G)$
of $G$-stable $\mbb{Z}_p$-lattices in semi-stable 
$p$-adic representations of $G$ with Hodge-Tate weights in $[0,r]$.
\end{theorem}

\section{Cartier duality}

Throughout this section,
for any $\mbb{Z}_p$-module $M$ and integer $n\ge 0$,
we put $M_n:=\mbb{Z}_p/p^n\mbb{Z}_p\otimes_{\mbb{Z}_p} M$ and 
$M_{\infty}:=\mbb{Q}_p/\mbb{Z}_p\otimes_{\mbb{Z}_p} M$.

\subsection{Duality on Kisin modules}
In this subsection, 
we recall Liu's results on duality theorems 
for Kisin modules
(\cite{Li1}, Section 3).
We start with the following example:

\begin{example}
Let $\mfS^{\vee}:=\mfS\cdot \mfrak{f}^r$ 
be the rank-$1$ free $\mfS$-module 
with $\vphi(\mfrak{f}^r):=c_0^{-r}E(u)^r\cdot \mfrak{f}^r$ where 
$pc_0$ is the constant coefficient of $E(u)$.
We denote by $\vphi^{\vee}$ this Frobenius $\vphi$.
Then $(\mfS^{\vee},\vphi^{\vee})$ is a free Kisin module of height $r$ and   
there exists an isomorphism 
$T_{\mfS}(\mfS^{\vee})\simeq \mbb{Z}_p(r)$
as $\mbb{Z}_p[G_{\infty}]$-modules
(see \cite{Li1}, Example 2.3.5).
Put $\mfS^{\vee}_{\infty}
:=\mbb{Q}_p/\mbb{Z}_p\otimes_{\mbb{Z}_p} \mfS^{\vee}
=\mfS_{\infty}\cdot \mfrak{f}^r$
(resp.\ $\mfS^{\vee}_{n}
:=\mbb{Z}_p/p^n\mbb{Z}_p\otimes_{\mbb{Z}_p} \mfS^{\vee}
=\mfS_{n}\cdot \mfrak{f}^r$ for any integer $n\ge 0$).
The Frobenius $\vphi$ on $\mfS^{\vee}$ induces 
Frobenii $\vphi^{\vee}$ on $\mfS_{\infty}^{\vee}$ 
and $\mfS_n^{\vee}$.

Put $\E^{\vee}:=\E\otimes_{\mfS} \mfS^{\vee}=\E\cdot \mfrak{f}^r$
and equip $\E^{\vee}$ with a Frobenius $\vphi^{\vee}$ 
arising from that of
$\E$ and $\mfS^{\vee}$.
Similarly, we put 
$\cO_{\E}^{\vee}=\cO_{\E}\cdot \mfrak{f}^r,
(\E/\cO_{\E})^{\vee}=\E/\cO_{\E}\cdot \mfrak{f}^r,
\cO_{\E,n}^{\vee}=\cO_{\E,n}\cdot \mfrak{f}^r$
and equip them with Frobenii $\vphi^{\vee}$
which arise from that of $\E^{\vee}$. 
We define 
$\cO^{\mrm{ur},\vee}$ and 
$\cO_{n}^{\mrm{ur},\vee}$, and Frobenii $\vphi^{\vee}$ on them by the analogous way.
\end{example}

Let $\mfM$ be a  Kisin module of height $r$
and denote by $M:=\cO_{\E}\otimes_{\mfS} \mfM$
the corresponding \'etale $\vphi$-module.
Put 
\[
\mfM^{\vee}:=\mrm{Hom}_{\mfS}(\mfM, \mfS_{\infty}),\
M^{\vee}:=\mrm{Hom}_{\cO_{\E},\vphi}(M,\E/\cO_{\E})
\quad \mrm{if}\ \mfM\ \mrm{is\ killed\ by\ some\ power\ of}\ p
\]
and 
\[
\mfM^{\vee}:=\mrm{Hom}_{\mfS}(\mfM, \mfS),\
M^{\vee}:=\mrm{Hom}_{\cO_{\E},\vphi}(M, \cO_{\E})
\quad \mrm{if}\ \mfM\ \mrm{is\ free}.
\]
We then have natural pairings
\[
\langle \cdot , \cdot \rangle \colon \mfM \times \mfM^{\vee}\to
\mfS_{\infty}^{\vee},\
\langle \cdot , \cdot \rangle \colon 
M\times M^{\vee}\to (\E/\cO_{\E})^{\vee}
\quad \mrm{if}\ \mfM\ \mrm{is\ killed\ by\ some\ power\ of}\ p
\]
and 
\[
\langle \cdot , \cdot \rangle \colon \mfM \times 
\mfM^{\vee}\to \mfS^{\vee},\
\langle \cdot , \cdot \rangle \colon 
M\times M^{\vee}\to \cO_{\E}^{\vee}
\quad \mrm{if}\ \mfM\ \mrm{is\ free}.
\]
The Frobenius $\vphi_{\mfM}^{\vee}$ on $\mfM^{\vee}$ 
(resp.\  $\vphi_{M}^{\vee}$ on $M^{\vee}$ )
is defined to be 
\[
\langle \vphi_{\mfM}(x), 
\vphi_{\mfM}^{\vee}(y) \rangle = 
\vphi^{\vee}(\langle x,y \rangle) 
\quad \mrm{for}\ x\in \mfM, y\in \mfM^{\vee}.
\]
\[
(\mrm{resp.}\ 
\langle \vphi_{M}(x), \vphi_{M}^{\vee}(y) \rangle = 
\vphi^{\vee}(\langle x,y \rangle) \quad \mrm{for}\ x\in M, y\in M^{\vee}.)
\]

\begin{theorem}[\cite{Li1}]
Let $\mfM$ be a torsion (resp.\ free) 
Kisin module of height $r$,
$M:=\cO_{\E}\otimes_{\mfS} \mfM$ 
the corresponding \'etale $\vphi$-module
and 
$\langle \cdot , \cdot \rangle$
the paring as above.

\noindent
$(1)$ $(\mfM^{\vee}, \vphi_{\mfM}^{\vee})$
is a torsion (resp.\ free) Kisin module of height $r$.
Similarly, 
$M^{\vee}$ is a torsion (resp.\ free) \'etale $\vphi$-module.

\noindent
$(2)$ 
A natural map $\cO_{\E}\otimes_{\mfS} \mfM^{\vee}\to M^{\vee}$
is an isomorphism and 
$\vphi^{\vee}_{M}=
\vphi_{\cO_{\E}}\otimes \vphi^{\vee}_{\mfM}$.

\noindent
$(3)$ All parings $\langle \cdot , \cdot \rangle$ 
appeared in the above are perfect.
\end{theorem}   
\begin{remark}
The assertion (2) of the above theorem says that there exists a natural isomorphism 
$\cO_{\E}\otimes_{\mfS} {\mfM}^{\vee}\simeq (\cO_{\E}\otimes_{\mfS} {\mfM})^{\vee}=M^{\vee}$
which is compatible with $\vphi$-structures.
In fact, the paring 
$\langle \cdot , \cdot \rangle$
for $M$ is equal to the pairing 
which is obtained by tensoring $\cO_{\E}$
to the pairing $\langle \cdot , \cdot \rangle$ for $\mfM$.    
\end{remark}

\subsection{Dual $(\varphi, \hat{G})$-modules}

In this subsection,
we construct dual $(\vphi,\hat{G})$-modules. 
Put
\[
\wh{\mcal{R}}^{\vee}
:=\wh{\mcal{R}}\otimes_{\vphi,\mfS} \mfS^{\vee}
=\wh{\mcal{R}}\otimes_{\vphi,\mfS} (\mfS\cdot \mfrak{f}^r)
=\wh{\mcal{R}}\cdot \mfrak{f}^r,
\]
\[
\wh{\mcal{R}}_n^{\vee}
:=\mbb{Z}_p/p^n\mbb{Z}_p\otimes_{\mbb{Z}_p} \wh{\mcal{R}}^{\vee}
=\wh{\mcal{R}}\otimes_{\vphi,\mfS} \mfS_n^{\vee}
=\wh{\mcal{R}}\otimes_{\vphi,\mfS} (\mfS_n\cdot \mfrak{f}^r)
=\wh{\mcal{R}}_n\cdot  \mfrak{f}^r
\quad \mrm{for\ any\ integer}\ n\ge 0
\]
and
\[
\wh{\mcal{R}}_{\infty}^{\vee}
:=\mbb{Q}_p/\mbb{Z}_p\otimes_{\mbb{Z}_p} \wh{\mcal{R}}^{\vee}
=\wh{\mcal{R}}\otimes_{\vphi,\mfS} \mfS_{\infty}^{\vee}
=\wh{\mcal{R}}\otimes_{\vphi,\mfS} (\mfS_{\infty}\cdot \mfrak{f}^r)
=\wh{\mcal{R}}_{\infty}\cdot  \mfrak{f}^r,
\]
and we equip them with natural Frobenii arising from those of $\wh{\mcal{R}}$ and $\mfS^{\vee}$.
By Theorem \ref{equiv},
we can define a unique $\hat{G}$-action on $\wh{\mcal{R}}^{\vee}$ such that 
$\wh{\mcal{R}}^{\vee}$ has a structure as a $(\vphi,\hat{G})$-module 
of height $r$ and there exists an isomorphism
\begin{equation}
\label{ga}
\hat{T}(\wh{\mcal{R}}^{\vee})\simeq \mbb{Z}_p(r)
\end{equation}
as $\mbb{Z}_p[G]$-modules.
This $\hat{G}$-action on $\wh{\mcal{R}}^{\vee}$ induces 
$\hat{G}$-actions on $\wh{\mcal{R}}_n^{\vee}$
and  $\wh{\mcal{R}}_{\infty}^{\vee}$. 
Then it is not difficult to see that $\wh{\mcal{R}}_n^{\vee}$ has a structure as a torsion 
$(\vphi,\hat{G})$-module of height $r$
and there exists an isomorphism
\begin{equation}
\label{ga2}
\hat{T}(\wh{\mcal{R}}_n^{\vee})\simeq \mbb{Z}_p/p^n\mbb{Z}_p(r)
\end{equation}
as $\mbb{Z}_p[G]$-modules.
We may say that $\wh{\mcal{R}}^{\vee}$ (resp.\ $\wh{\mcal{R}}_n^{\vee}$ ) 
is a dual $(\vphi,\hat{G})$-module
of $\wh{\mcal{R}}$ (resp.\ $\wh{\mcal{R}}_n$ )
since (\ref{ga}) and (\ref{ga2}) hold.

\begin{remark}
If 
$K_{p^{\infty}}\cap K_{\infty}=K$ (which is automatically hold in the case $p>2$), 
then $\hat{G}$-actions on $\wh{\mcal{R}}^{\vee},\wh{\mcal{R}}_n^{\vee}$
and $\wh{\mcal{R}}_{\infty}^{\vee}$  
can be written explicitly as follows
(see Example 3.2.3 of \cite{Li3}):
If $K_{p^{\infty}}\cap K_{\infty}=K$,
we have $\hat{G}=G_{p^{\infty}}\rtimes H_K$ (see Lemma 5.1.2 in \cite{Li2}). 
Fixing a topological generator $\tau\in G_{p^{\infty}}$,
we define $\hat{G}$-actions on the above three modules 
by the relation $\tau(\mfrak{f}^r):=\hat{c}^r\cdot \mfrak{f}^r$.
Here $\hat{c}:=\frac{c}{\tau(c)},\ 
c:=\prod^{\infty}_{n=0}\vphi^n(\frac{\vphi(c_0^{-1}E(u))}{p})$.
Example 3.2.3 of \cite{Li3} says that $c\in A_{\mrm{cris}}^{\times}$ 
and $\hat{c}\in \wh{\mcal{R}}^{\times}$.
It follows from straightforward calculations  that 
$\wh{\mcal{R}}^{\vee}$ and 
$\wh{\mcal{R}}_n^{\vee}$ 
are $(\vphi,\hat{G})$-modules of height $r$.
\end{remark}

\begin{lemma}
\label{lemma}
Let $A$ be a $\mfS$-algebra 
with characteristic coprime to $p$.
Let $\mfM$ be a torsion 
(resp.\ free) Kisin module.  
Then there exist natural isomorphisms:
\[
A\otimes_{\vphi, \mfS} \mfM^{\vee} 
\overset{\simeq}{\longrightarrow} 
\mrm{Hom}_A(A\otimes_{\vphi, \mfS} \mfM, A_{\infty})\quad 
\mrm{if}\ \mfM\ \mrm{is\ killed\ by\ some\ power\ of}\ p,
\]
\[
A\otimes_{\vphi, \mfS} \mfM^{\vee} 
\overset{\simeq}{\longrightarrow} 
\mrm{Hom}_A(A\otimes_{\vphi, \mfS} \mfM, A)\quad 
\mrm{if}\ \mfM\ \mrm{is\ free}.
\]
\end{lemma}
\begin{proof}
If $\mfM$ is free, the statement is clear.
If $p\mfM=0$, then we may regard $\mfM$ as a 
finite free $\mfS/p\mfS$-module 
and thus the statement is clear.
Suppose that $\mfM$ is a 
(general) torsion Kisin module
of height $r$. 
By Proposition 2.3.2 of \cite{Li1},
there exists an extension of 
$\vphi$-modules
\[
0=\mfM_0\subset \mfM_1\subset \cdots \subset \mfM_n=\mfM 
\]
such that, for all $1\le i\le n$,
$\mfM_i/\mfM_{i-1}\in \mrm{Mod}^{r,\mrm{tor}}_{/\mfS}$ and 
$\mfM_i/\mfM_{i-1}$ is a finite free $\mfS/p\mfS=k[\![u]\!]$-module.
Furthermore, we have $\mfM_i\in \mrm{Mod}^{r,\mrm{tor}}_{/\mfS}$
by Lemma 2.3.1 in \cite{Li1}. 
We show that the natural map
\[
A\otimes_{\vphi, \mfS} \mfM_i^{\vee} 
\longrightarrow
\mrm{Hom}_A(A\otimes_{\vphi, \mfS} 
\mfM_i, A_{\infty}),\quad
a\otimes f\mapsto (a\otimes x\mapsto af(x))
\]
where $a\in A, f\in \mfM_i^{\vee}$ and  $x\in \mfM_i$,
is an isomorphism by induction for $i$.
For $i=0$, it is obvious. 
Suppose that the above map is an isomorphism for $i-1$. 
We have an exact sequence of $\mfS$-modules
\begin{equation}
\label{ex1}
0\to \mfM_{i-1}\to \mfM_i\to \mfM_i/\mfM_{i-1}\to 0. 
\end{equation}
By Corollary 3.1.5 of \cite{Li1}, we know that 
the sequence
\[
0\to (\mfM_i/\mfM_{i-1})^{\vee}\to 
\mfM_i^{\vee}\to \mfM_{i-1}^{\vee}\to 0. 
\]
is also an exact sequence of $\mfS$-modules.
Therefore, we have the following 
exact sequence of $A$-modules:
\begin{equation}
\label{ex2}
A\otimes_{\vphi,\mfS}(\mfM_i/\mfM_{i-1})^{\vee}\to 
A\otimes_{\vphi,\mfS}\mfM_i^{\vee}\to 
A\otimes_{\vphi,\mfS}\mfM_{i-1}^{\vee}\to 0. 
\end{equation}
On the other hand, the exact sequence  
(\ref{ex1}) induces an exact sequence of $A$-modules
\begin{equation}
\label{ex3}
0\to 
\mrm{Hom}_{A}(A\otimes_{\vphi, \mfS} \mfM_i/\mfM_{i-1}, A_{\infty})\to 
\mrm{Hom}_{A}(A\otimes_{\vphi, \mfS} \mfM_i, A_{\infty})\to 
\mrm{Hom}_{A}(A\otimes_{\vphi, \mfS} \mfM_{i-1}, A_{\infty}).
\end{equation}
Combining sequences (\ref{ex2}) and (\ref{ex3}), 
we obtain the following commutative 
diagram of A-modules:
\begin{center}
$\displaystyle \xymatrix{
 & A\otimes_{\vphi,\mfS}(\mfM_i/\mfM_{i-1})^{\vee}\ar[r] \ar[d] & 
A\otimes_{\vphi,\mfS}\mfM_i^{\vee}\ar[r] \ar[d]
& A\otimes_{\vphi,\mfS}\mfM_{i-1}^{\vee}\ar[r] \ar[d]& 0 \\
0\ar[r] & \mrm{Hom}_A(A\otimes_{\vphi, \mfS} \mfM_i/\mfM_{i-1}, A_{\infty})\ar[r] & 
\mrm{Hom}_A(A\otimes_{\vphi, \mfS} \mfM_i, A_{\infty})\ar[r]
& \mrm{Hom}_A(A\otimes_{\vphi, \mfS} \mfM_{i-1}, A_{\infty})  &  }$
\end{center}
where the two rows are exact. 
Furthermore, 
first and third columns are isomorphisms
by the induction hypothesis.
By the snake lemma, 
we obtain that the second column is an isomorphism, too.
\end{proof}

Let $\hat{\mfM}=(\mfM,\vphi_{\mfM},\hat{G})$ be a torsion 
(resp.\ free) $(\vphi,\hat{G})$-module of height $r$
and $(\mfM^{\vee},\vphi^{\vee}_{\mfM})$ 
the dual Kisin module of $(\mfM,\vphi_{\mfM})$.
By Lemma \ref{lemma},
we have isomorphisms
\begin{equation}
\label{act1}
\wh{\mcal{R}}\otimes_{\vphi, \mfS} \mfM^{\vee} 
\overset{\simeq}{\longrightarrow} 
\mrm{Hom}_{\wh{\mcal{R}}}(\wh{\mcal{R}}\otimes_{\vphi, \mfS} \mfM, \wh{\mcal{R}}_{\infty}^{\vee})\quad 
\mrm{if}\ \mfM\ \mrm{is\ killed\ by\ some\ power\ of}\ p
\end{equation} 
(resp.\ 
\begin{equation}
\label{act2}
\wh{\mcal{R}}\otimes_{\vphi, \mfS} \mfM^{\vee} 
\overset{\simeq}{\longrightarrow} 
\mrm{Hom}_{\wh{\mcal{R}}}(\wh{\mcal{R}}\otimes_{\vphi, \mfS} \mfM, \wh{\mcal{R}}^{\vee})\quad 
\mrm{if}\ \mfM\ \mrm{is\ free}).
\end{equation}  
We define $\hat{G}$-action on
$\mrm{Hom}_{\wh{\mcal{R}}}(\wh{\mcal{R}}\otimes_{\vphi, \mfS} \mfM, \wh{\mcal{R}}_{\infty}^{\vee})$
(resp.\ 
$\mrm{Hom}_{\wh{\mcal{R}}}(\wh{\mcal{R}}\otimes_{\vphi, \mfS} \mfM, \wh{\mcal{R}}^{\vee})$)
by
\[
(\sigma.f)(x):=\sigma(f(\sigma^{-1}(x)))
\]
for $\sigma\in \hat{G}, x\in \wh{\mcal{R}}\otimes_{\vphi,\mfS}\mfM$ and 
$f\in \mrm{Hom}_{\wh{\mcal{R}}}(\wh{\mcal{R}}\otimes_{\vphi, \mfS} \mfM, \wh{\mcal{R}}_{\infty}^{\vee})$
(resp.\ $f\in \mrm{Hom}_{\wh{\mcal{R}}}(\wh{\mcal{R}}\otimes_{\vphi, \mfS} \mfM, \wh{\mcal{R}}^{\vee}))$
and equip $\hat{G}$-action on $\wh{\mcal{R}}\otimes_{\vphi, \mfS} \mfM^{\vee}$
via an isomorphism (\ref{act1}) (resp.\ (\ref{act2})).

\begin{theorem}
\label{dual}
Let $\hat{\mfM}=(\mfM,\vphi_{\mfM},\hat{G})$ 
be a torsion (resp.\ free) $(\vphi,\hat{G})$-module 
of height $r$
and equip  $\hat{G}$-action on 
$\wh{\mcal{R}}\otimes_{\vphi,\mfS} \mfM^{\vee}$ as the above.
Then the triple 
$\hat{\mfM}^{\vee}:=(\mfM^{\vee}, \vphi^{\vee}_{\mfM}, \hat{G})$ 
is a torsion (resp.\ free) 
$(\vphi,\hat{G})$-module of height $r$.
\end{theorem}

\begin{definition}
We call $\hat{\mfM}^{\vee}$ as Theorem \ref{dual} 
the {\it dual $(\vphi,\hat{G})$-module of $\hat{\mfM}$}.
\end{definition}

To prove Theorem \ref{dual}, 
we need the following easy property for $\wh{\mcal{R}}_{\infty}=\wh{\mcal{R}}[1/p]/\wh{\mcal{R}}$. 

\begin{lemma}
\label{lemm}
$(1)$ For any integer $n$, we have 
\[
\whR[1/p]\cap p^nW(\mrm{Fr}R)=\whR\cap p^nW(R)=p^n\whR.
\]
\noindent
$(2)$ 
The following properties for an $a\in \whR[1/p]$ are equivalent:

$(\mrm{i})$ If $x\in \whR[1/p]$ satisfies that $ax=0$ in $\whRi$, then $x=0$ in $\whRi$.

$(\mrm{ii})$ $a\notin p\whR$.

$(\mrm{iii})$ $a\notin pW(R)$. 

$(\mrm{iv})$ $a\notin pW(\mrm{Fr}R)$. 

\end{lemma}

\begin{proof}
(1) The result follows from the relations
\[
\whR[1/p]\cap p^nW(\mrm{Fr}R)=
\whR[1/p]\cap (W(R)[1/p]\cap p^nW(\mrm{Fr}R))=
\whR[1/p]\cap p^nW(R)
\]
and 
\[
p^n\whR\subset \whR[1/p]\cap p^nW(R)\subset \whR_{K_0}\cap p^nW(R)=
p^n(\whR_{K_0}\cap W(R))=p^n\whR.
\]

\noindent
(2) The equivalence of (ii), (iii) and (iv) follows from the assertion (1).
Suppose the condition (iv) holds.
Take any $x\in \whR[1/p]$ such that $ax\in \whR$.
Then we have
\[
\frac{1}{a}\whR\cap \whR[1/p]\subset 
\frac{1}{a}W(\mrm{Fr}R)\cap W(\mrm{Fr}R)[1/p]\subset W(\mrm{Fr}R)
\]
since $a\notin pW(\mrm{Fr}R)$.
Thus we obtain
\[
x\in \frac{1}{a}\whR\cap \whR[1/p]=
\frac{1}{a}\whR\cap \whR[1/p]\cap W(\mrm{Fr}R)\subset 
\whR[1/p]\cap W(\mrm{Fr}R)=\whR,
\]
which implies the assertion (i) (the last equality follows from (1)).
Suppose the condition (ii) does not hold, that is, $a\in p\whR$.
Then $\whR[1/p]\cap \frac{1}{a}\whR\supset \frac{1}{a}\whR\supset 
\frac{1}{p}\whR\supsetneq \whR$ and this implies that (i) does not hold.
\end{proof}

\begin{proof}[Proof of Theorem \ref{dual}]
We only prove the case where $\hat{\mfM}$ is a torsion  $(\vphi,\hat{G})$-module
(the free case can be checked by almost all the same method).

We check the properties $(1)$ to $(5)$ 
of Definition \ref{Liu} for $\hat{\mfM}^{\vee}$.
It is clear that $(1)$ and $(2)$ hold for $\hat{\mfM}^{\vee}$.
Take any $f\in \mfM^{\vee}$.
Regard $\mfM^{\vee}$ 
as a submodule of $\wh{\mcal{R}}\otimes_{\vphi, \mfS} \mfM^{\vee}$.
Then, in $\wh{\mcal{R}}\otimes_{\vphi, \mfS} \mfM^{\vee}$, 
we see that $f$ is equal to the map 
\[
\hat{f}\colon \wh{\mcal{R}}\otimes_{\vphi, \mfS} \mfM\to \wh{\mcal{R}}\cdot \mfrak{f}^r
\]
given by $a\otimes x\mapsto a\vphi(f(x))\cdot \mfrak{f}^r$ 
for $a\in \wh{\mcal{R}}$ and $x\in \mfM$.
Since $\mfM\subset (\wh{\mcal{R}}\otimes_{\vphi,\mfS} \mfM)^{H_K}$, 
we have 
\begin{align*}
(\sigma.\hat{f})(a\otimes x)&=\sigma(\hat{f}(\sigma^{-1}(a\otimes x)))
=\sigma(\hat{f}(\sigma^{-1}(a)(1\otimes x)))
=\sigma((\sigma^{-1}(a)\hat{f}(1\otimes x)))\\
&=a\sigma(\hat{f}(1\otimes x))
=a\sigma(\vphi(f(x))\cdot \mfrak{f}^r)
=a\vphi(f(x))\cdot \mfrak{f}^r
=\hat{f}(a\otimes x).
\end{align*}
for any $a\in \wh{\mcal{R}}, x\in \mfM$ and $\sigma\in H_K$.
This implies $\mfM^{\vee}\subset (\wh{\mcal{R}}\otimes_{\vphi,\mfS} \mfM^{\vee})^{H_K}$
and hence $(4)$ holds for $\hat{\mfM}^{\vee}$.
Check the property $(5)$, that is, the condition that 
$\hat{G}$ acts trivially on $\hat{\mfM}/I_+\hat{\mfM}$.
By Lemma \ref{lemma}, we know that 
there exists the following natural isomorphism:
\[
\wh{\mcal{R}}\otimes_{\vphi, \mfS} \mfM^{\vee}/
I_+(\wh{\mcal{R}}\otimes_{\vphi, \mfS} \mfM^{\vee}) 
\overset{\simeq}{\longrightarrow} 
\mrm{Hom}_{\wh{\mcal{R}}}
(\wh{\mcal{R}}\otimes_{\vphi, \mfS} \mfM/
I_+(\wh{\mcal{R}}\otimes_{\vphi, \mfS} \mfM), 
\wh{\mcal{R}}^{\vee}_{\infty}/I_+\wh{\mcal{R}}^{\vee}_{\infty}),
\]
which is in fact $\hat{G}$-equivalent by the definition of $\hat{G}$-action on 
$\wh{\mcal{R}}\otimes_{\vphi, \mfS} \mfM^{\vee}$.
Since $\hat{G}$ acts on $\wh{\mcal{R}}\otimes_{\vphi, \mfS} \mfM/
I_+(\wh{\mcal{R}}\otimes_{\vphi, \mfS} \mfM)$ and 
$\wh{\mcal{R}}^{\vee}_{\infty}/I_+\wh{\mcal{R}}^{\vee}_{\infty}$
trivially,
we obtain the desired result.

Finally we prove the property $(3)$ for $\hat{\mfM}^{\vee}$.
First we note that, if we take any  $f\in \mfM^{\vee}=\mrm{Hom}_{\mfS}(\mfM, \mfS_{\infty})$
and regard $f$ as a map which has values in $\mfS_{\infty}^{\vee}$, 
then we have 
\begin{equation}
\label{aaa}
\vphi^{\vee}(f)\circ \vphi_{\mfM}=\vphi^{\vee}\circ f\colon \mfM\to \mfS_{\infty}^{\vee}.
\end{equation}
Recall that there exists a natural isomorphism
\[
\wh{\mcal{R}}\otimes_{\vphi, \mfS} \mfM^{\vee} 
\simeq \mrm{Hom}_{\wh{\mcal{R}}}(\wh{\mcal{R}}\otimes_{\vphi, \mfS} \mfM, \wh{\mcal{R}}_{\infty}^{\vee})
\]
by Lemma \ref{lemma}. 
We equip $\mrm{Hom}_{\wh{\mcal{R}}}(\wh{\mcal{R}}\otimes_{\vphi, \mfS} \mfM, \wh{\mcal{R}}_{\infty}^{\vee})$
with a $\vphi$-structure $\vphi^{\vee}$ 
via this isomorphism.
Then it is enough to show that 
$\sigma \vphi^{\vee}=\vphi^{\vee}\sigma$ on 
$\mrm{Hom}_{\wh{\mcal{R}}}(\wh{\mcal{R}}\otimes_{\vphi, \mfS} \mfM, \wh{\mcal{R}}_{\infty}^{\vee})$ 
for any $\sigma\in \hat{G}$.
Take any 
$\hat{f}\in \mrm{Hom}(\wh{\mcal{R}}\otimes_{\vphi, \mfS} \mfM, \wh{\mcal{R}}_{\infty}^{\vee})$
and consider the following diagram:
\begin{equation}
\label{com2}
\displaystyle \xymatrix{
  \wh{\mcal{R}}\otimes_{\vphi,\mfS} \mfM 
  \ar^{\vphi_{\hat{\mfM}}}[r] \ar_{\hat{f}}[d]
& \wh{\mcal{R}}\otimes_{\vphi,\mfS} \mfM 
  \ar^{\vphi^{\vee}(\hat{f})}[d]\\
  \wh{\mcal{R}}_{\infty}^{\vee}
  \ar_{\vphi^{\vee}}[r]
& \wh{\mcal{R}}_{\infty}^{\vee}
. }
\end{equation}
By (\ref{aaa}), 
we obtain that the diagram (\ref{com2}) is also commutative. 
To check the relation $\sigma(\vphi^{\vee}(\hat{f}))=\vphi^{\vee}(\sigma(\hat{f}))$,
it suffices to show that 
$\sigma(\vphi^{\vee}(\hat{f}))(\vphi_{\hat{\mfM}}(x))=
\vphi^{\vee}(\sigma(\hat{f}))(\vphi_{\hat{\mfM}}(x))$
for any $x\in \wh{\mcal{R}}\otimes_{\vphi, \mfS} \mfM$ since 
$\mfM$ is of finite $E(u)$-height and, for any $a\in \wh{\mcal{R}}_{\infty}$, 
$\vphi(E(u))a=0$ if and only if $a=0$ by Lemma \ref{lemm}.
By (\ref{com2}), we have
\[
\sigma(\vphi^{\vee}(\hat{f}))(\vphi_{\hat{\mfM}}(x))
=\sigma(\vphi^{\vee}(\hat{f})(\sigma^{-1}(\vphi_{\hat{\mfM}}(x))))
=\sigma(\vphi^{\vee}(\hat{f})(\vphi_{\hat{\mfM}}(\sigma^{-1}(x))))
=\sigma(\vphi^{\vee}(\hat{f}(\sigma^{-1}(x)))).
\]
By replacing $\hat{f}$ with $\sigma(\hat{f})$ in the diagram (\ref{com2}), we have 
\[
\vphi^{\vee}(\sigma(\hat{f}))(\vphi_{\hat{\mfM}}(x))
=\vphi^{\vee}(\sigma(\hat{f}))(x)
=\vphi^{\vee}(\sigma(\hat{f}(\sigma^{-1}(x))
=\sigma(\vphi^{\vee}(\hat{f}(\sigma^{-1}(x))))
\]
and this finishes the proof.
\end{proof}

\begin{corollary}
Let $\hat{\mfM}$ be a $(\vphi,\hat{G})$-module.
Then the natural map of $(\vphi, \hat{G})$-modules
\begin{equation}
\label{self}
\hat{\mfM}\to (\hat{\mfM}^{\vee})^{\vee},\quad x\in \mfM\mapsto (f\mapsto f(x))\in (\mfM^{\vee})^{\vee}
\end{equation}
is an isomorphism of $(\vphi,\hat{G})$-modules.
\end{corollary}
\begin{proof}
It is known that the map (\ref{self}) 
is an isomorphism by Proposition 3.1.7 in \cite{Li1}.
Hence it is enough to check that 
the map (\ref{self}) is compatible 
with Galois action after tensoring $\wh{\mcal{R}}$, but
it follows from straightforward calculations. 
\end{proof}

Since the assignment 
$\hat{\mfrak{M}}\mapsto \hat{\mfrak{M}}^{\vee}$
is a functor from the category of 
torsion (resp.\ free) $(\vphi,\hat{G})$-modules of height $r$ 
to itself,
we obtain the following.
\begin{corollary}
The assignment 
$\hat{\mfrak{M}}\mapsto \hat{\mfrak{M}}^{\vee}$
is an anti-equivalence on the category 
of torsion (resp.\ free)
$(\vphi, \hat{G})$-modules.
A quasi-inverse is given by
$\hat{\mfrak{M}}\mapsto \hat{\mfrak{M}}^{\vee}$.
\end{corollary}

\subsection{Compatibility with Galois actions}
The goal of this subsection is to prove the following:
\begin{theorem}
\label{Gal}
Let $\hat{\mfM}$ be a $(\vphi,\hat{G})$-module.
Then we have
\begin{equation}
\hat{T}(\hat{\mfM}^{\vee})\simeq \hat{T}^{\vee}(\hat{\mfM})(r)
\end{equation}
as $\mbb{Z}_p[G]$-modules 
where $\hat{T}^{\vee}(\hat{\mfM})$ is the dual representation of 
$\hat{T}(\hat{\mfM})$ and
the symbol ``$(r)$'' is for the $r$-th Tate twist.
\end{theorem}

First we construct a covariant functor
for the category of $(\vphi, \hat{G})$-modules. 
Recall that, if $\hat{\mfM}=(\mfM,\vphi_{\mfM},\hat{G})$ 
is a $(\vphi, \hat{G})$-module,
we often abuse of notations by 
denoting $\hat{\mfM}$ the underlying 
module $\wh{\mcal{R}}\otimes_{\vphi, \mfS} \mfM$.
\begin{proposition}
\label{prop1}
Let $\hat{\mfM}$ be a $(\vphi, \hat{G})$-module.
Then the natural $W(\mrm{Fr}R)$-linear map
\begin{equation}
\label{fund1}
W(\mrm{Fr}R)\otimes_{\mbb{Z}_p}
(W(\mrm{Fr}R)\otimes_{\wh{\mcal{R}}} \hat{\mfM})^{\vphi=1}
\to 
W(\mrm{Fr}R)\otimes_{\wh{\mcal{R}}} 
\hat{\mfM},\quad a\otimes x\mapsto ax, 
\end{equation}
for any  $a\in W(\mrm{Fr}R)$ and $x\in(W(\mrm{Fr}R)\otimes_{\wh{\mcal{R}}} \hat{\mfM})^{\vphi=1}$,
is an isomorphism, which is compatible 
with $\vphi$-structures and $G$-actions.
\end{proposition}

\begin{proof}
A non-trivial assertion of this proposition is 
only the bijectivity of the map (\ref{fund1}).
First we note the following natural $\vphi$-equivariant isomorphisms:
\begin{align*}
        W(\mrm{Fr}R)\otimes_{\wh{\mcal{R}}} \hat{\mfM}
&\simeq W(\mrm{Fr}R)\otimes_{\vphi,\mfS} \mfM \\
&\simeq W(\mrm{Fr}R)\otimes_{\cO_{\E}}(\cO_{\E}\otimes_{\vphi,\mfS}M)\\
&\overset{1\otimes \vphi^{\ast}_M}{\longrightarrow} W(\mrm{Fr}R)\otimes_{\cO_{\E}} M
\end{align*}
where $M:=\cO_{\E}\otimes_{\mfS} \mfM$ 
is the \'etale $\vphi$-module corresponding to $\mfM$.
Here the bijectivity of $1\otimes \vphi^{\ast}_M$, where $\vphi^{\ast}_M$ 
is the $\cO_{\E}$-linearization of $\vphi_M$, 
follows from the \'etaleness of $M$. 
Combining the above isomorphisms and the relation (\ref{Fon1}),
we obtain the following natural $\vphi$-equivalent bijective  maps
\begin{equation}
\label{hosoku}
W(\mrm{Fr}R)\otimes_{\wh{\mcal{R}}} \hat{\mfM}
\overset{\simeq}{\longrightarrow}  
W(\mrm{Fr}R)\otimes_{\cO_{\E}} M
\overset{\simeq}{\longleftarrow}
W(\mrm{Fr}R)\otimes_{\mbb{Z}_p}
(\cO^{\mrm{ur}}\otimes_{\cO_{\E}} M)^{\vphi=1} 
\end{equation}
and hence we obtain 
\begin{equation}
\label{is1}
(W(\mrm{Fr}R)\otimes_{\wh{\mcal{R}}} \hat{\mfM})^{\vphi=1}
\simeq
(\cO^{\mrm{ur}}\otimes_{\cO_{\E}} M)^{\vphi=1}. 
\end{equation}
By (\ref{hosoku}) and (\ref{is1}),
we obtain an isomorphism
\[
W(\mrm{Fr}R)\otimes_{\mbb{Z}_p}
(W(\mrm{Fr}R)\otimes_{\wh{\mcal{R}}} \hat{\mfM})^{\vphi=1}
\overset{\simeq}{\longrightarrow} 
W(\mrm{Fr}R)\otimes_{\wh{\mcal{R}}} \hat{\mfM}
\]  
and the desired result follows from the fact that this isomorphism coincides with the natural map (\ref{fund1}).
\end{proof}

For any  $(\vphi, \hat{G})$-module $\hat{\mfM}$,
we set
\[
\hat{T}_{\ast}(\hat{\mfM}):=(W(\mrm{Fr}R)\otimes_{\hat{\mcal{R}}}\hat{\mfM})^{\vphi=1}.
\]
Since the Frobenius action on 
$W(\mrm{Fr}R)\otimes_{\hat{\mcal{R}}}\hat{\mfM}$ 
commutes with $G$-action,
we see that $G$ acts on  $\hat{T}_{\ast}(\hat{\mfM})$ stable. 
We have shown in the proof of Proposition \ref{prop1} (see (\ref{is1}))
that 
\[
\hat{T}_{\ast}(\hat{\mfM})\simeq T_{\ast}(M)
\]
as $\mbb{Z}_p[G_{\infty}]$-modules
for $M:=\cO_{\e}\otimes_{\mfS} \mfM$ (recall that the functor $T_{\ast}$ is defined in Section 2.2).
In particular, 
if $\hat{\mfM}$ is free and $d:=\mrm{rank}_{\mfS}(\mfM)$, 
$\hat{T}_{\ast}(\hat{\mfM})$ is free of 
rank $d$ as a $\mbb{Z}_p$-module.
The association $\hat{\mfM}\mapsto \hat{T}_{\ast}(\hat{\mfM})$ 
is a covariant functor from the category of 
$(\vphi, \hat{G})$-modules of height $r$
to the category $\mrm{Rep}_{\mbb{Z}_p}(G)$ 
of finite $\mbb{Z}_p[G]$-modules.
By the exactness of the functor $T_{\ast}$,
the functor $\hat{T}_{\ast}$ is an exact functor.

\begin{corollary}
\label{cov}
The $\mbb{Z}_p$-representation $\hat{T}_{\ast}(\hat{\mfM})$ of $G$ is the dual of 
$\hat{T}(\hat{\mfM})$, that is,
\[
\hat{T}^{\vee}(\hat{\mfM})\simeq \hat{T}_{\ast}(\hat{\mfM})
\]
as $\mbb{Z}_p[G]$-modules. 
\end{corollary}

\begin{proof}
Suppose $\hat{\mfM}$ is killed by some power of $p$.
By Proposition \ref{prop1} and 
the relation $W(\mrm{Fr}R)_{\infty}^{\vphi=1}=\mbb{Q}_p/\mbb{Z}_p$,
we have 
\begin{align*}
\mrm{Hom}_{\mbb{Z}_p}(\hat{T}_{\ast}(\hat{\mfM}),\mbb{Q}_p/\mbb{Z}_p)
&\simeq 
\mrm{Hom}_{W(\mrm{Fr}R),\vphi}
(W(\mrm{Fr}R)\otimes_{\mbb{Z}_p}
(W(\mrm{Fr}R)\otimes_{\wh{\mcal{R}}} \hat{\mfM})^{\vphi=1},
W(\mrm{Fr}R)_{\infty})\\
&\simeq
\mrm{Hom}_{W(\mrm{Fr}R),\vphi}(
W(\mrm{Fr}R)\otimes_{\wh{\mcal{R}}} \hat{\mfM},
W(\mrm{Fr}R)_{\infty})\\
&\simeq
\mrm{Hom}_{\wh{\mcal{R}},\vphi}(
\hat{\mfM},
W(\mrm{Fr}R)_{\infty})=\hat{T}(\hat{\mfM}).
\end{align*}
The last equality follows from 
the proof of Lemma 3.1.1 of \cite{Li3},
but we include
a proof here for the sake of completeness.
Take any 
$h\in \mrm{Hom}_{\wh{\mcal{R}},\vphi}(
\hat{\mfM},
W(\mrm{Fr}R)_{\infty})$.
It is enough to prove that 
$h$ has in fact values in $W(R)_{\infty}$.
Put $g:=h|_{\mfM}$.
Since $g$ is a $\vphi(\mfS)$-linear morphism 
from $\mfM$ to $W(R)_{\infty}=\vphi(W(R)_{\infty})$,
there exists a $\mfS$-linear morphism 
$\mfrak{g}\colon \mfM\to W(\mrm{Fr}R)_{\infty}$ 
such that $\vphi(\mfrak{g})=g$.
Furthermore, we see that $\mfrak{g}$ is $\vphi$-equivariant.
Note that $\mfrak{g}(\mfM)\subset W(\mrm{Fr}R)_{\infty}$ 
is a $\mfS$-finite type $\vphi$-stable submodule
and of $E(u)$-height $r$.
By \cite{Fo}, Proposition B.1.8.3, we have 
$\mfrak{g}(\mfM)\subset \mfS_{\infty}^{\mrm{ur}}$. 
Since 
\[
h(a\otimes x)=a\vphi(\mfrak{g}(x))
\]
for any $a\in \wh{\mcal{R}}$ and $x\in \mfM$, we obtain that 
$h$ has values in $W(R)_{\infty}$.

The case $\hat{\mfM}$ is free, 
we obtain the desired result 
by the same proof as above 
if we replace $W(\mrm{Fr}R)_{\infty}$ 
(resp.\ $\mbb{Q}_p/\mbb{Z}_p$) with 
$W(\mrm{Fr}R)$ (resp.\ $\mbb{Z}_p$). 
\end{proof}

In the rest of this subsection, 
we prove Theorem \ref{Gal}.
We only prove the case where 
$\mfM$ is killed by $p^n$ for some integer $n\ge 1$
(we can prove the free case 
by an analogous way and the free case is 
easier than the torsion case).

First we consider natural pairings
\begin{equation}
\label{pair1}
\langle \cdot , \cdot \rangle \colon \mfM \times \mfM^{\vee}\to
\mfS_n^{\vee}
\end{equation}
and 
\begin{equation}
\label{pair2}
\langle \cdot , \cdot \rangle \colon M\times M^{\vee}\to \cO_{\E,n}^{\vee}
\end{equation}
which are perfect and compatible with $\vphi$-structures.
Here $M:=\cO_{\E}\otimes_{\mfS} \mfM$ is the \'etale $\vphi$-module
corresponding to $\mfM$.
We can extend the pairing (\ref{pair2}) to the $\vphi$-equivalent perfect pairing  
\[
(\wh{\cO^{\mrm{ur}}}\otimes_{\cO_{\E}} M)
\times (\wh{\cO^{\mrm{ur}}}\otimes_{\cO_{\E}}M^{\vee})
\to \cO_n^{\mrm{ur},\vee}.
\]
Since the above pairing is $\vphi$-equivariant and 
$(\cO_n^{\mrm{ur},\vee})^{\vphi=1}\simeq \mbb{Z}_p/p^n\mbb{Z}_p(-r)$,
we have a pairing 
\begin{equation}
\label{pair23}
(\wh{\cO^{\mrm{ur}}}\otimes_{\cO_{\E}} M)^{\vphi=1}
\times (\wh{\cO^{\mrm{ur}}}\otimes_{\cO_{\E}}M^{\vee})^{\vphi=1}
\to \mbb{Z}_p/p^n\mbb{Z}_p(-r)
\end{equation}
compatible with $G_{\infty}$-actions.
Liu showed in the proof of Lemma 3.1.2 in \cite{Li1} 
that this pairing is perfect.
By similar way, we have the following paring
\begin{equation}
\label{pair233}
(W(\mrm{Fr}R)\otimes_{\cO_{\E}} M)^{\vphi=1}
\times (W(\mrm{Fr}R)\otimes_{\cO_{\E}}M^{\vee})^{\vphi=1}
\to \mbb{Z}_p/p^n\mbb{Z}_p(-r).
\end{equation} 
On the other hand,
the pairing (\ref{pair1}) induces a pairing 
\begin{equation}
\label{pair3}
(\wh{\mcal{R}}\otimes_{\vphi, \mfS}\mfM) 
\times (\wh{\mcal{R}}\otimes_{\vphi, \mfS}\mfM^{\vee})\to
\wh{\mcal{R}}_n^{\vee}.
\end{equation}
We can extend the pairing (\ref{pair3}) to the $\vphi$-equivalent perfect pairing
\[
(W(\mrm{Fr}R)\otimes_{\wh{\mcal{R}}}(\wh{\mcal{R}}\otimes_{\vphi, \mfS}\mfM)) 
\times (W(\mrm{Fr}R)\otimes_{\wh{\mcal{R}}}(\wh{\mcal{R}}\otimes_{\vphi, \mfS}\mfM^{\vee}))\to
W(\mrm{Fr}R)\otimes_{\wh{\mcal{R}}} \wh{\mcal{R}}_n^{\vee}.
\]
Since the above pairing is $\vphi$-equivariant and 
$(W(\mrm{Fr}R)\otimes_{\wh{\mcal{R}}} \wh{\mcal{R}}_n^{\vee})^{\vphi=1}\simeq \mbb{Z}_p/p^n\mbb{Z}_p(-r)$,
we have a pairing 
\begin{equation}
\label{pair4}
(W(\mrm{Fr}R)\otimes_{\wh{\mcal{R}}}(\wh{\mcal{R}}\otimes_{\vphi, \mfS}\mfM))^{\vphi=1} 
\times (W(\mrm{Fr}R)\otimes_{\wh{\mcal{R}}}(\wh{\mcal{R}}\otimes_{\vphi, \mfS}\mfM^{\vee}))^{\vphi=1}\to
\mbb{Z}_p/p^n\mbb{Z}_p(-r)
\end{equation}
compatible with $G$-actions.
Since we have the natural isomorphism 
$\wh{\cO^{\mrm{ur}}}\otimes_{\mbb{Z}_p}(\wh{\cO^{\mrm{ur}}}\otimes_{\cO_{\E}} M)^{\vphi=1}
\overset{\simeq}{\longrightarrow}\wh{\cO^{\mrm{ur}}}\otimes_{\cO_{\E}} M$,
we obtain the $\vphi$-equivariant isomorphisms
\begin{equation}
\label{isom1}
W(\mrm{Fr}R)\otimes_{\wh{\mcal{R}}}\hat{\mfM} \overset{\simeq}{\longrightarrow}
W(\mrm{Fr}R)\otimes_{\cO_{\E}} M \overset{\simeq}{\longleftarrow}
W(\mrm{Fr}R)\otimes_{\mbb{Z}_p}(\wh{\cO^{\mrm{ur}}}\otimes_{\cO_{\E}} M)^{\vphi=1}.
\end{equation}
Therefore, combining (\ref{pair23}), (\ref{pair233}), (\ref{pair4}) and (\ref{isom1}), 
we have the following diagram
\begin{center}
$\displaystyle \xymatrix{
  \! \! \!(W(\mrm{Fr}R)\otimes_{\wh{\mcal{R}}}\hat{\mfM})^{\vphi=1}\! \! \!
  \ar[d]_{\simeq} 
& \! \! \! \times \! \! \!
& \! \! \!(W(\mrm{Fr}R)\otimes_{\wh{\mcal{R}}}\hat{\mfM}^{\vee})^{\vphi=1}
  \ar[d]_{\simeq} \ar[r]
& \mbb{Z}_p/p^n\mbb{Z}_p(-r)
  \ar@{=}[d] \\
  \! \! \!(W(\mrm{Fr}R)\otimes_{\cO_{\E}}M)^{\vphi=1}\! \! \!
& \! \! \! \times \! \! \!
& \! \! \!(W(\mrm{Fr}R)\otimes_{\cO_{\E}}M^{\vee})^{\vphi=1}
  \ar[r]
& \mbb{Z}_p/p^n\mbb{Z}_p(-r) \\
  (\wh{\cO^{\mrm{ur}}}\otimes_{\cO_{\E}} M)^{\vphi=1}
  \ar[u]^{\simeq}
& \! \! \! \times \! \! \!
& \! \! \!(\wh{\cO^{\mrm{ur}}}\otimes_{\cO_{\E}} M^{\vee})^{\vphi=1}
  \ar[u]^{\simeq} \ar[r]
& \mbb{Z}_p/p^n\mbb{Z}_p(-r) 
\ar@{=}[u]
. }$
\end{center}
It is a straightforward calculation to check that 
the above diagram is commutative.
Since the bottom pairing is perfect, 
we see that the top pairing is also perfect.
This implies $\hat{T}_{\ast}(\hat{\mfM}^{\vee})\simeq \hat{T}_{\ast}(\hat{\mfM})(-r)$ and 
therefore, we have the desired result by Corollary \ref{cov}.

\begin{remark}
A triple $\hat{\mfM}=(\mfM,\vphi,\hat{G})$ 
is called a {\it weak $(\vphi,\hat{G})$-module}
if it only satisfies axioms (1), (2), (3) and (4) in Definition \ref{Liu}.
Weak $(\vphi,\hat{G})$-modules are related with potentially semistable representations.
By the same proofs, all results in this section hold if we replace ``$(\vphi,\hat{G})$-modules''
with ``weak $(\vphi,\hat{G})$-modules''.
\end{remark}

\subsection{Comparisons with Breuil modules}
Throughout this subsection,
we suppose $p\ge 3$ and $r<p-1$.
In this situation, Liu showed in \cite{Li2} that 
there exists a contravariant functor $T_{\mrm{st}}$ 
from the category $\mrm{Mod}^r_{/S}$ of strongly divisible lattices of weight $r$
into the category $\mrm{Rep}^r_{\mbb{Z}_p}(G)$ which gives an 
equivalence of those categories.
Hence there exists an equivalence of categories between $\mrm{Mod}^r_{/S}$ 
and $\mrm{Mod}^{r,\hat{G},\mrm{fr}}_{/\mfS}$ (see also Theorem \ref{equiv} (2)). 
Liu's arguments in Section 5 of \cite{Li2} give an explicit correspondence
between the objects of those categories as below.

Let $\hat{\mfM}$ be a free $(\vphi,\hat{G})$-module of height $r$ and put $T:=\hat{T}(\hat{\mfM})$.
Then $V:=T\otimes_{\mbb{Z}_p} \mbb{Q}_p$ is a semistable $p$-adic representation of $G$ with 
Hodge-Tate weights in $[0,r]$
and $\mathscr{D}:=S_{K_0}\otimes_{\vphi, \mfS} \mfM$ has 
a structure as a  Breuil module 
corresponding\footnote{
Breuil showed, in Th\'eor\`eme 6.1.1 of \cite{Br}, an equivalence of categories between 
the category of semistable representations of $G$ and 
the category of Breuil modules.} 
to $V$.
If we put $\mcal{M}:=S\otimes_{\vphi, \mfS} \mfM\subset \mathscr{D}$,
then $\mcal{M}$ has a structure as a quasi-strongly divisible 
lattice in $\mathscr{D}$. 
Liu showed in Lemma 3.5.3 of \cite{Li2} that
$\mcal{M}$ is in fact automatically stable under the monodromy operator
$N_{\mathscr{D}}$ of the Breuil module $\mathscr{D}$ and 
thus $\mcal{M}$ has a structure as a strongly divisible 
lattice in $\mathscr{D}$. 
Moreover, $T_{\mrm{st}}(\mcal{M})$ is isomorphic to $T$.
On the other hands, 
we can realize the monodromy operator $N_{\mathscr{D}}$
via the action of $\tau$ on $\hat{\mfM}$. 
To see this, 
we define a natural $G$-action on $B^{+}_{\mrm{cris}}\otimes_{S_{K_0}} \mathscr{D}$ as follows:
For any $\sigma\in G$ and $a\otimes x\in B^{+}_{\mrm{cris}}\otimes_{S_{K_0}} \mathscr{D}$,
define
\[
\sigma(a\otimes x)=\sum^{\infty}_{i=0} \sigma(a)\gamma_i
(-\mrm{log}([{\underline{\e}}(\sigma)]))\otimes N^i_{\mathscr{D}}(x),
\]
where ${\underline{\e}}(\sigma):=(\sigma(\pi_n)/\pi_n)_{n\ge 0}\in R$.
This action is a well-defined $B^{+}_{\mrm{cris}}$-semi-linear $G$-action on 
$B^{+}_{\mrm{cris}}\otimes_{S_{K_0}} \mathscr{D}$ (\cite{Li2}, Lemma 5.1.1). 
From now on, 
we fix $t:=-\mrm{log}([{\underline{\e}}(\tau)])$.
Then, for any $n\ge 0$ and $x\in \mathscr{D}$, 
an induction on $n$  shows that 
\[
(\tau-1)^n(x)=\sum^{\infty}_{m=n}({\sum_{i_1+i_2+\cdots i_n=m, i_j\ge 1}
{\frac{m!}{i_1!\cdots i_n!}}})\gamma_m(t)\otimes N_{\mathscr{D}}^m(x).
\]
Hence if we put 
$\mrm{log}(\tau)(x):=\sum^{\infty}_{n=1}(-1)^{n-1}\frac{(\tau-1)^n}{n}(x)$ 
for any $x\in \mathscr{D}$,
we have 
\[
\mrm{log}(\tau)(x)=t\otimes N_{\mathscr{D}}(x).
\]
Therefore, we obtain the well-defined functor
\[
\mcal{M}_{\whR}\colon \mrm{Mod}^{r,\hat{G},\mrm{fr}}_{/\mfS}\to \mrm{Mod}^r_{/S},\quad 
\hat{\mfM}\to S\otimes_{\vphi, \mfS} \mfM  
\]
which makes the following diagram commutative:
\begin{center}
$\displaystyle \xymatrix{  
  \mrm{Mod}^{r,\hat{G},\mrm{fr}}_{/\mfS}\ar_{\hat{T}}^{\simeq}[r] 
  \ar_{\mcal{M}_{\whR}}[d]
& \mrm{Rep}^r_{\mbb{Z}_p}(G) 
  \ar @{=}[d]\\
  \mrm{Mod}^r_{/S}
  \ar_{T_{\mrm{st}}}^{\simeq}[r]
& \mrm{Rep}^r_{\mbb{Z}_p}(G)
. }$
\end{center}
Here we equip $\mcal{M}_{\whR}(\hat{\mfM})=S\otimes_{\vphi, \mfS} \mfM$ 
with the following additional structures;
an $S$-submodule $\mrm{Fil}^r(\mcal{M}_{\whR}(\hat{\mfM}))$
of $\mcal{M}_{\whR}(\hat{\mfM})$ defined by
\[
\mrm{Fil}^r(\mcal{M}_{\whR}(\hat{\mfM})):=
\{x\in \mcal{M}_{\whR}(\hat{\mfM})| 
(1\otimes \vphi)(x)\in \mrm{Fil}^rS\otimes_{\mfS}\mfM \},
\]
a $\vphi_{S}$-semi-linear endomorphism 
$\vphi_r\colon \mrm{Fil}^r(\mcal{M}_{\whR}(\hat{\mfM}))\to \mcal{M}_{\whR}(\hat{\mfM})$ 
defined by the composition
\[
\mrm{Fil}^r(\mcal{M}_{\whR}(\hat{\mfM}))
\overset{1\otimes \vphi}{\longrightarrow} 
\mrm{Fil}^rS\otimes_{\mfS}\mfM
\overset{\frac{1}{p^r}\vphi\otimes 1}{\longrightarrow} 
S\otimes_{\vphi, \mfS} \mfM= \mcal{M}_{\whR}(\hat{\mfM})
\] 
and a monodromy operator $N$ on $\mcal{M}_{\whR}(\hat{\mfM})$
given by $N=\frac{1}{t}\mrm{log}(\tau)$.
Since the above diagram is commutative, 
we see that the functor $\mcal{M}_{\whR}$ 
gives an equivalence of categories.
By the Galois compatibility of a duality (cf.\ Theorem \ref{Gal}), 
we obtain 
\begin{corollary}
Let $\hat{\mfrak{M}}$ be a  free 
$(\varphi, \hat{G})$-module.
Then there exists an isomorphism 
\[
\mcal{M}_{\whR}(\hat{\mfM}^{\vee})\simeq
\mcal{M}_{\whR}(\hat{\mfM})^{\vee},
\]
which is functorial for $\hat{\mfM}$.
Here  $\mcal{M}_{\whR}(\hat{\mfM})^{\vee}$ is the Cartier dual of 
the strongly divisible lattice $\mcal{M}_{\whR}(\hat{\mfM})$
(cf.\ \cite{Ca2}, Chapter V).
\end{corollary}

\end{document}